\documentclass[10pt]{amsart}
\usepackage{amssymb,MnSymbol,times}
\usepackage{amsthm,amsmath,cite}
\usepackage{color}

\title[A generalization of the concept of distance]{A generalization of the concept of distance based on the simplex inequality}\thanks{Corresponding author: Bruno Teheux is with the Mathematics Research Unit, University of Luxembourg, Maison du Nombre, 6, avenue de la Fonte, L-4364 Esch-sur-Alzette, Luxembourg. Email: bruno.teheux[at]uni.lu}

\author{Gergely Kiss}
\address{Mathematics Research Unit, University of Luxembourg, Maison du Nombre, 6, avenue de la Fonte, L-4364 Esch-sur-Alzette, Luxembourg}
\email{gergely.kiss[at]uni.lu}

\author{Jean-Luc Marichal}
\address{Mathematics Research Unit, University of Luxembourg, Maison du Nombre, 6, avenue de la Fonte, L-4364 Esch-sur-Alzette, Luxembourg}
\email{jean-luc.marichal[at]uni.lu}

\author{Bruno Teheux}
\address{Mathematics Research Unit, University of Luxembourg, Maison du Nombre, 6, avenue de la Fonte, L-4364 Esch-sur-Alzette, Luxembourg}
\email{bruno.teheux[at]uni.lu}

\date{January 11, 2018}

\begin{document}

\theoremstyle{plain}
\newtheorem{theorem}{Theorem}[section]
\newtheorem{lemma}[theorem]{Lemma}
\newtheorem{proposition}[theorem]{Proposition}
\newtheorem{corollary}[theorem]{Corollary}
\newtheorem{fact}[theorem]{Fact}
\newtheorem*{main}{Main Theorem}

\theoremstyle{definition}
\newtheorem{definition}[theorem]{Definition}
\newtheorem{example}[theorem]{Example}
\newtheorem{algorithm}{Algorithm}

\theoremstyle{remark}
\newtheorem*{conjecture}{Conjecture}
\newtheorem{remark}{Remark}
\newtheorem{claim}{Claim}

\newenvironment{disarray}%
 {\everymath{\displaystyle\everymath{}}\array}%
 {\endarray}

\newcommand{\N}{\mathbb{N}}
\newcommand{\R}{\mathbb{R}}
\newcommand{\Q}{\mathbb{Q}}
\newcommand{\Vspace}{\vspace{2ex}}
\newcommand{\bfx}{\mathbf{x}}
\newcommand{\bfy}{\mathbf{y}}
\newcommand{\bfz}{\mathbf{z}}
\newcommand{\me}{\mathrm{\mathbf{m}}}

\begin{abstract}
We introduce and discuss the concept of \emph{$n$-distance}, a generalization to $n$ elements of the classical notion of distance obtained by replacing the triangle inequality with the so-called simplex inequality
\[
d(x_1, \ldots, x_n)~\leq~K\, \sum_{i=1}^n d(x_1, \ldots, x_n)_i^z{\,}, \qquad x_1, \ldots, x_n, z \in X,
\]
where $K=1$. Here $d(x_1,\ldots,x_n)_i^z$ is obtained from the function $d(x_1,\ldots,x_n)$ by setting its $i$th variable to $z$. We provide several examples of $n$-distances, and for each of them we investigate the infimum of the set of real numbers $K\in\left]0,1\right]$ for which the inequality above holds. We also introduce a generalization of the concept of $n$-distance obtained by replacing in the simplex inequality the sum function with an arbitrary symmetric function.
\end{abstract}



\maketitle

\section{Introduction}

The notion of metric space, as first introduced by Fr\'echet \cite{Frechet1906} and later developed by Hausdorff \cite{Hausdorff1914}, is one of the key ingredients in many areas of pure and applied mathematics, particularly in analysis, topology, geometry, statistics, and data analysis.

Denote the half-line $\left[0,+\infty\right[$ by $\R_+$. Recall that a \emph{metric space} is a pair $(X,d)$, where $X$ is a nonempty set and $d$ is a distance on $X$, that is, a function $d\colon X^2\to\R_+$ satisfying the following conditions:

\begin{itemize}
\item $d(x_1,x_2) \leq d(x_1,z)+d(z,x_2)$ for all $x_1,x_2,z\in X$ (triangle inequality),
\item $d(x_1,x_2)=d(x_2,x_1)$ for all $x_1,x_2\in X$ (symmetry),
\item $d(x_1,x_2)=0$ if and only if $x_1=x_2$ (identity of indiscernibles).
\end{itemize}

Generalizations of the concept of distance in which $n\geq 3$ elements are considered have been investigated by several authors (see, e.g., \cite[Chapter 3]{Deza2014} and the references therein). The three conditions above may be generalized to $n$-variable functions $d\colon X^n\to\R_+$ in the following ways. For any integer $n\geq 1$, we set $[n]=\{1,\ldots,n\}$. For any $i\in [n]$ and any $z\in X$, we denote by $d(x_1,\ldots,x_n)_i^z$ the function obtained from $d(x_1,\ldots,x_n)$ by setting its $i$th variable to $z$. Let also denote by $S_n$ the set of all permutations on $[n]$. A function $d\colon X^n \to \R_+$ is said to be an \emph{$(n-1)$-semimetric} \cite{Deza2000} if it satisfies

\begin{enumerate}
\item[(i)] $d(x_1,\ldots,x_n)\leq \sum_{i=1}^n d(x_1,\ldots,x_n)_i^z$ for all $x_1,\ldots,x_n,z\in X$,

\item[(ii)] $d(x_1,\ldots,x_n)=d(x_{\pi(1)},\ldots,x_{\pi(n)})$ for all $x_1,\ldots,x_n\in X$ and all $\pi\in S_n$,
\end{enumerate}
and it is said to be an \emph{$(n-1)$-hemimetric} \cite{Deza2014,Deza2003} if additionally it satisfies
\begin{enumerate}
\item[(iii')] $d(x_1,\ldots,x_n)=0$ if and only if $x_1,\ldots,x_n$ are not pairwise distinct.
\end{enumerate}

Condition (i) is referred to as the \emph{simplex inequality} \cite{Deza2014,Deza2000}. For $n=3$, this inequality can be interpreted as follows: the area of a triangle face of a tetrahedron does not exceed the sum of the areas of the remaining three faces.

The following variant of condition (iii') can also be naturally considered:

\begin{enumerate}
\item[(iii)] $d(x_1, \ldots, x_n)=0$ if and only if $x_1=\cdots=x_n$.
\end{enumerate}

For $n=3$, functions satisfying conditions (i), (ii), and (iii) were introduced by Dhage~\cite{Dhage1992} and called \emph{$D$-distances}. Their topological properties were investigated subsequently \cite{Dhage1994,Dhage1994bis,Dhage2000}, but unfortunately most of the claimed results are incorrect, see \cite{Mustafa2004}. Moreover, it turned out that a stronger version of $D$-distance is needed for a sound topological use of these functions \cite{Mustafa2004,Mustafa2006,Khamsi2015}.

In this paper we introduce and discuss the following simultaneous generalization of the concepts of distance and $D$-distance by considering functions with $n\geq 2$ arguments.

\begin{definition}[{see \cite{KisMarTeh16}}]\label{de:n-dist}
Let $n\geq 2$ be an integer. We say that $(X,d)$ is an \emph{$n$-metric space} if $X$ is a nonempty set and $d$ is an \emph{$n$-distance} on $X$, that is, a function $d\colon X^n\to\R_+$ satisfying conditions (i), (ii), and (iii).
\end{definition}

We observe that for any $n$-distance $d\colon X^n \to \R_+$, the set of real numbers $K\in\left]0,1\right]$ for which the condition
\begin{equation}\label{eq:ntr}
d(x_1, \ldots, x_n)~\leq~K\, \sum_{i=1}^n d(x_1, \ldots, x_n)_i^z{\,},\qquad x_1,\ldots,x_n,z\in X,
\end{equation}
holds has an infimum $K^*$. We call it the \emph{best constant} associated with the $n$-distance $d$. Determining the value of $K^*$ for a given $n$-distance is an interesting problem that might be mathematically challenging. It is the purpose of this paper to provide natural examples of $n$-distances and to show how elegant the investigation of the values of the best constants might be.

It is worth noting that determining the best constant $K^*$ is not relevant for nonconstant $(n-1)$-hemimetrics because we always have $K^*=1$ for those functions. Indeed, we have
$$
0 ~<~ d(x_1, \ldots, x_n) ~=~ \sum_{i=1} ^ n d(x_1, \ldots, x_n)_i^{x_n}
$$
for any pairwise distinct elements $x_1, \ldots, x_n$ of $X$.

The paper is organized as follows. In Section 2 we provide some basic properties of $n$-metric spaces as well as some examples of $n$-distances together with their corresponding best constants. In Section 3 we investigate the values of the best constants for Fermat point based $n$-distances and discuss the particular case of median graphs. In Section 4 we consider some geometric constructions (smallest enclosing sphere and number of directions) to define $n$-distances and study their corresponding best constants. In Section 5 we introduce a generalization of the concept of $n$-distance by replacing in condition (i) the sum function with an arbitrary symmetric $n$-variable function. Finally, in Section 6 we conclude the paper by proposing topics for further research.

\begin{remark}\label{rem:multdist}
A \emph{multidistance} on $X$, as introduced by Mart\'{\i}n and Mayor \cite{MarMay11}, is a function $d\colon\bigcup_{n\geqslant 1}X^n\to\R_+$ such that, for every integer $n\geq 1$, the restriction of $d$ to $X^n$ satisfies conditions (ii), (iii), and
\begin{enumerate}
\item[(i')] \label{it:multi03} $d(x_1,\ldots,x_n) \leqslant \sum_{i=1}^n d(x_i,z)$ for all $x_1,\ldots,x_n,z\in X$.
\end{enumerate}
Properties of multidistances as well as instances including the Fermat multidistance and smallest enclosing ball multidistances have been investigated for example in \cite{MarMay10,MarMay11,MarMayVal11,AguMarMaySunVal12}. Note that multidistances have an indefinite number of arguments whereas $n$-distances have a fixed number of arguments. In particular, an $n$-distance can be defined without referring to any given $2$-distance. Interestingly, some of the $n$-distances we present in this paper cannot be constructed from the concept of multidistance (see Section 6).
\end{remark}

\section{Basic examples and general properties of $n$-distances}

Let us illustrate the concept of $n$-distance by giving a few elementary examples. Other classes of $n$-distances will be investigated in the next sections. We denote by $|E|$ the cardinalily of any set $E$.

\begin{example}[Drastic $n$-distance]\label{ex:dra}
For every integer $n\geqslant 2$, the map $d\colon X^n\to\R_+$ defined by $d(x_1,\ldots,x_n)=0$, if $x_1=\cdots =x_n$, and $d(x_1,\ldots,x_n)=1$, otherwise, is an $n$-distance on $X$ for which the best constant is $K_n^*=\frac{1}{n-1}$. Indeed, let $x_1,\ldots,x_n,z\in X$ and assume that $d(x_1,\ldots,x_n)=1$. If there exists $k\in [n]$ such that $x_i=x_j\neq x_k$ for all $i,j\in [n]\setminus\{k\}$, then we have
$$
\sum_{i=1}^n d(x_1, \ldots, x_n)_i^z ~=~
\begin{cases}
n-1, & \text{if $z\in\{x_1,\ldots,x_n\}\setminus\{x_k\}$},\\
n, & \text{otherwise}.
\end{cases}
$$
In all other cases we have $\sum_{i=1}^n d(x_1, \ldots, x_n)_i^z=n$.\qed
\end{example}

\begin{example}[Cardinality based $n$-distance]\label{ex:car}
For every integer $n\geqslant 2$, the map $d\colon X^n\to\R_+$ defined by
$$
d(x_1,\ldots,x_n) ~=~ |\{x_1,\ldots,x_n\}|-1
$$
is an $n$-distance on $X$ for which the best constant is $K_n^*=\frac{1}{n-1}$. Indeed, let $x_1,\ldots,x_n,z\in X$ and assume that $d(x_1,\ldots,x_n)\geq 1$. The case $n=2$ is trivial. So let us further assume that $n\geq 3$. For every $i\in [n]$, set $m_i=|\{j\in [n]\mid x_j=x_i\}|$. If $|\{x_1,\ldots,x_n\}|<n$ (which means that there exists $j\in [n]$ such that $m_j\geq 2$), then it is straightforward to see that
\begin{eqnarray*}
\sum_{i=1}^n d(x_1, \ldots, x_n)_i^z &\geq & n{\,}d(x_1,\ldots,x_n) - |\{i\in [n]\mid m_i=1\}|\\
& \geq & (n-1){\,}d(x_1,\ldots,x_n),
\end{eqnarray*}
where the first inequality is an equality if and only if $z=x_j$ for some $j\in [n]$ such that $m_j\geq 2$, and the second inequality is an equality if and only if there is exactly one $j\in [n]$ such that $m_j\geq 2$. If $|\{x_1,\ldots,x_n\}|=n$, then
$$
\sum_{i=1}^n d(x_1, \ldots, x_n)_i^z ~\geq ~ (n-1){\,}d(x_1,\ldots,x_n),
$$
with equality if and only if $z\in\{x_1,\ldots,x_n\}$.\qed
\end{example}

\begin{example}[Diameter]\label{ex:max}
Given a metric space $(X,d)$ and an integer $n\geqslant 2$, the map $d_{\max}\colon X^n \to \R_+$ defined by
$$
d_{\max}(x_1,\ldots,x_n)~=~\max_{\{i,j\}\subseteq [n]}d(x_i,x_j)
$$
is an $n$-distance on $X$ for which we have $K_n^*=\frac{1}{n-1}$. Indeed, let $x_1,\ldots,x_n,z\in X$ and assume without loss of generality that $d_{\max}(x_1, \ldots, x_n)=d(x_1, x_2)$. For every $i\in [n]$ we have
\[
d_{\max}(x_1, \ldots, x_n)_i^z  ~\geq ~
\begin{cases}
d(x_2,z), & \text{if $i=1$},\\
d(x_1,z), & \text{if $i=2$},\\
d(x_1,x_2), & \text{otherwise}.
\end{cases}
\]
Using the triangle inequality, we then obtain
\begin{align*}
\sum_{i=1}^n d_{\max}(x_1, \ldots, x_n)_i^z & ~\geq ~ (n-2)\, d(x_1,x_2)+d(x_1,z)+d(x_2,z)\\
& ~\geq ~ (n-1)\, d(x_1,x_2)~=~(n-1)\, d_{\max}(x_1, \ldots, x_n),
\end{align*}
which proves that $K_n^* \leq \frac{1}{n-1}$. To prove that $K_n^*=\frac{1}{n-1}$, note that if $x_1=\cdots=x_{n-1}=z$ and $x_n\neq z$, then $\sum_{i=1}^n d_{\max}(x_1, \ldots, x_n)_i^z=(n-1)\,  d_{\max}(x_1, \ldots, x_n)$.\qed
\end{example}

\begin{example}[Sum based $n$-distance]\label{ex:sum} Given a metric space $(X,d)$ and an integer $n\geq 2$, the map $d_\Sigma\colon X^n \to \R_+$ defined by
$$
d_{\Sigma}(x_1, \ldots, x_n)~=~\sum_{\{i,j\}\subseteq [n]} d(x_i, x_j)
$$
is an $n$-distance on $X$ for which we have $K_n^*=\frac{1}{n-1}$. Indeed, for fixed $x_1, \ldots, x_n, z\in X$, we have
$$
\sum_{i=1}^n d_\Sigma(x_1, \ldots, x_n)_i^z ~=~  (n-2)\, \sum_{\{i,j\}\subseteq [n]} d(x_i, x_j) + (n-1)\, \sum_{i=1}^n d(x_i,z).
$$
Using the triangle inequality we obtain
$$
(n-1)\, \sum_{i=1}^n d(x_i,z)~=~ \sum_{\{i,j\}\subseteq [n]} \big(d(x_i,z)+d(x_j,z)\big)~\geq~\sum_{\{i,j\}\subseteq [n]} d(x_i,x_j).
$$
Therefore, we finally obtain
$$
\sum_{i=1}^n d_\Sigma(x_1, \ldots, x_n)_i^z ~\geq~(n-1)\, \sum_{\{i,j\}\subseteq [n]} d(x_i, x_j)~=~(n-1)\, d_\Sigma(x_1, \ldots, x_n),
$$
which proves that $K_n^* \leq \frac{1}{n-1}$. To prove that $K_n^*=\frac{1}{n-1}$, note that if $x_1=\cdots=x_{n-1}=z$ and $x_n\neq z$, then $\sum_{i=1}^n d_{\Sigma}(x_1, \ldots, x_n)_i^z=(n-1)\,  d_{\Sigma}(x_1, \ldots, x_n)$.\qed
\end{example}

\begin{example}[Arithmetic mean based $n$-distance]\label{ex:AM}
For any integer $n\geq 2$, the map $d\colon \R^n \to \R_+$ defined by
\[
d(x_1, \ldots, x_n) ~=~ \frac{1}{n} \sum_{i=1}^n x_i - x_{(1)}~=~ \frac{1}{n} \sum_{i=1}^n (x_i - x_{(1)})\, ,
\]
where $x_{(1)}=\min\{x_1, \ldots, x_n\}$, is an $n$-distance on $\R$ for which $K_n^*=\frac{1}{n-1}$. Indeed, let $x_1, \ldots, x_n, z \in \R$. By symmetry of $d$ we may assume that $x_1\leq \cdots \leq x_n$. We then obtain
$$
d(x_1, \ldots, x_n)~=~\frac{1}{n}\,\Big(\sum_{i=1}^n x_i\Big)-x_1
$$
and
$$
\sum_{i=1}^n d(x_1, \ldots, x_n)_{i}^z~=~\big(1-\frac{1}{n}\big)\Big(\sum_{i=1}^n x_i\Big)+z-(n-1)\, \min\{x_1, z\}-\min\{x_2,z\}.
$$
It follows that condition \eqref{eq:ntr} holds for $K_n=\frac{1}{n-1}$ if and only if
$$
(n-1)(x_1-\min\{x_1,z\})+(z-\min\{x_2,z\})~\geq ~0.
$$
We then observe that this inequality is trivially satisfied, which proves that $K_n^* \leq \frac{1}{n-1}$. To prove that $K_n^*=\frac{1}{n-1}$, just take $x_1, \ldots, x_n, z\in\R$ so that $x_1< z< x_2=\cdots = x_n$.\qed
\end{example}

In the next result, we show how to construct an $(n-1)$-hemimetric from an $n$-distance.

\begin{proposition}
Let $(X,d)$ be an $n$-metric space for some integer $n\geq 2$. The function $d'\colon X^{n}\to \R_+$ defined as
\[
d'(x_1, \ldots, x_n)~=~
\begin{cases}
0, & \text{ if $x_1, \ldots, x_n$ are not pairwise distinct,}\\
d(x_1, \ldots, x_n), & \text{otherwise},
\end{cases}
\]
is an $(n-1)$-hemimetric.
\end{proposition}

\begin{proof}
It is easy to see that $d'$ satisfies conditions (ii) and (iii'). To see that condition (i) holds, let $x_1, \ldots, x_n, z \in X$ and assume that $d'(x_1, \ldots, x_n)>0$. If $d'(x_1, \ldots, x_n)_i^z=d(x_1, \ldots, x_n)_i^z$ for every $i\in [n]$, then the simplex inequality holds for $d'$. Otherwise, we must have $z\in\{x_1, \ldots, x_n\}$ and then $\sum_{i=1}^n d'(x_1, \ldots, x_n)_i^z=d'(x_1, \ldots, x_n)$. This shows that condition (i) holds.
\end{proof}

The next proposition shows that two of the standard constructions of distances from existing ones are still valid for $n$-distances. The proof uses the following lemma.

\begin{lemma}\label{lem:sum}
For any $a_1, \ldots, a_n, a\in \R_+$ such that $a\leq \sum_{i=1}^n a_i $, we have
\[
\frac{a}{1+a}~\leq ~ \sum_{i=1}^n \frac{a_i}{1+a_i}\, .
\]
\end{lemma}

\begin{proof}
We proceed by induction on $n\geq 1$. The result is easily obtained for $n\in \{1,2\}$. Assume that the result holds for $k\in \{1, \ldots, n-1\}$ for some $n\geq 3$, and that $a\leq \sum_{i=1}^{n} a_i$ for some $a, a_1, \ldots, a_n \in \R_+$. Letting $b=\max\{0, a-a_n\}$, we obtain
\[
\frac{a}{1+a}~\leq ~\frac{b}{1+b} + \frac{a_n}{1+a_n}~\leq ~\sum_{i=1}^n \frac{a_i}{1+a_i}\, ,
\]
where the first inequality is obtained by the induction hypothesis applied to $a\leq b+a_n$, and the second to $b\leq \sum_{i=1}^{n-1} a_i$.
\end{proof}

\begin{proposition}\label{prop:buit-d}
Let $d$ and $d'$ be $n$-distances on $X$ and let $\lambda >0$. The following assertions hold.
\begin{enumerate}
\item[(a)] $d+d'$ and $\lambda{\,}d$ are $n$-distances on $X$.
\item[(b)] $\frac{d}{1+d}$ is an $n$-distance on $X$, with values in $[0,1]$.
\end{enumerate}
\end{proposition}

\begin{proof}
(a) is a simple verification. For (b) we note that condition (i) holds for $\frac{d}{1+d}$ by Lemma~\ref{lem:sum}.
\end{proof}


\begin{remark}
In the same spirit as Proposition~\ref{prop:buit-d} we observe that if $d\colon X\to \R_+$ is an $n$-distance and $d_0\colon X\to \R_+$  is an $(n-1)$-hemimetric, then $d+d_0$ is an $n$-distance.
\end{remark}

\section{Fermat point based $n$-distances}\label{sec:fermat}

Recall that, given a metric space $(X,d)$ and an integer $n\geq 2$, the \emph{Fermat set} $F_Y$ of any $n$-element subset $Y=\{x_1, \ldots, x_n\}$ of $X$ is defined as
\[
F_Y~=~\Big\{x \in X ~\Big|~ \sum_{i=1}^n d(x_i,x)~\leq ~ \sum_{i=1}^n d(x_i,z) \text{ for all } z \in X\Big\}.
\]
Elements of $F_Y$ are the \emph{Fermat points of $Y$}. The problem of finding the Fermat point of a triangle in the Euclidean plane was formulated by Fermat in the early 17th century, and was first solved by Torricelli around 1640. The general problem stated for $n\geq 2$ in any metric space was considered by many authors, and applications were found for instance in geometry, combinatorial optimization, and facility location. We refer to \cite[Chapter II]{BolMarSol99} and \cite{Drezner2002} for an account of the history of this problem. Also, in \cite{JahKupMarRic15}, the location problem is extended in various directions and studied also for very general metrics -- more general than those of normed spaces.

We observe that $F_Y$ need not be nonempty in a general metric space. However, it follows from the continuity of the function $h\colon X \to \R_+$ defined by $h(x)=\sum_{i=1}^n d(x_i,x)$ that $F_Y$ is nonempty whenever $(X,d)$ is a proper metric space. (Recall that a metric space is proper if every closed ball is compact.) In this section we will therefore assume that $(X,d)$ is a proper metric space.

\begin{proposition}\label{prop:fermat}
For any proper metric space $(X,d)$ and any integer $n\geq 2$, the map $d_F\colon X^n \to \R_+$ defined as
\[
d_F(x_1, \ldots, x_n)~=~\min_{x \in X}\sum_{i=1}^n d(x_i,x){\,},
\]
is an $n$-distance on $X$ and we call it the \emph{Fermat $n$-distance}.
\end{proposition}

\begin{proof}
The map $d_F$ clearly satisfies conditions (ii) and (iii). Let us show that it satisfies condition (i). Assume first that $n=2$ and let $y_1,y_2\in X$ be such that
$$
d_F(z,x_2) ~=~ d(z,y_1)+d(x_2,y_1)\quad\text{and}\quad d_F(x_1,z) ~=~ d(x_1,y_2)+d(z,y_2).
$$
By applying the triangle inequality, we obtain
\begin{eqnarray*}
d_F(z,x_2)+d_F(x_1,z) &=& (d(x_1,y_2)+d(z,y_2))+(d(z,y_1)+d(x_2,y_1)) \\
& \geq & d(x_1,x_2) ~=~ d(x_1,x_1)+d(x_1,x_2) ~\geq ~ d_F(x_1,x_2).
\end{eqnarray*}

Assume now that $n\geq 3$ and let $y_1, \ldots, y_n \in X$ be such that
\[
d_F(x_1, \ldots, x_n)_{i}^{z}~=~\sum_{\substack{ j\neq i}} d(x_j,y_i)+d(z,y_i) , \qquad i=1,\ldots,n.
\]
It follows that
\begin{align*}
 \sum_{i=1}^n d_F(x_1, \ldots, x_n)_{i}^{z}&   ~\geq ~\sum_{i=1}^n \sum_{\substack{ j \neq i}} d(x_j, y_i)\\
& ~\geq ~\big(d(x_1,y_n)+d(x_2,y_n)\big)+\sum_{i=2}^{n-1} \big(d(x_{1},y_i)+d(x_{i+1},y_{i})\big),
\end{align*}
that is, by applying the triangle inequality,
\[
 \sum_{i=1}^n d_F(x_1, \ldots, x_n)_{i}^{z} ~\geq ~ \sum_{i=2}^{n} d(x_1, x_i)~=~ \sum_{i=1}^{n} d(x_1, x_i)~\geq ~d_F(x_1, \ldots, x_n),
\]
where the last inequality follows from the definition of $d_F$.
\end{proof}

In the next proposition we use rough counting arguments to obtain bounds for the best constant $K_n^*$ associated with the Fermat $n$-distance.

\begin{proposition}\label{prop:gen-F-K}
For every $n\geq 2$, the best constant $K_n^*$ associated with the Fermat $n$-distance satisfies the inequalities $\frac{1}{n-1}\le K_n^*\le \frac{1}{\lfloor n/2\rfloor}$.
\end{proposition}

\begin{proof}
Let $x_1, \ldots, x_n \in X$ and let $z$ be a Fermat point of $\{x_1, \ldots, x_n\}$. For every $i\in [n]$, denote by $y_i$ a Fermat point of $\{z\} \cup \{x_1, \ldots, x_n\}\setminus\{x_i\}$. We then have
\begin{align}
d_F(x_1, \ldots, x_n)_i^z & ~=~\sum_{j\neq i}d(x_j, y_i)+d(z,y_i)\label{eq:lmsngf0}\\
& ~\leq~\sum_{j\neq i}d(x_j, z)+d(z,z)~=~\sum_{j\neq i}d(x_j, z).\nonumber
\end{align}
By summing over $i=1,\ldots,n$, we obtain
\[
\sum_{i=1}^n d_F(x_1, \ldots, x_n)_i^z~\leq~(n-1)\, \sum_{i=1}^n d(x_i,z)~=~(n-1)\, d_F(x_1, \ldots, x_n),
\]
which shows that $K_n^*\geq 1/(n-1)$.

Now, if $z$ denotes any element of $X$ and if $y_1, \ldots, y_n$ are defined as in the first part of the proof, the identity \eqref{eq:lmsngf0} holds for every $i\in [n]$. Then, for $i=1,\ldots,n-1$, we have
\begin{align}
d_F(x_1, \ldots, x_n)_i^z+d_F(x_1, \ldots, x_n)_{i+1}^z  & ~\geq~d(z,y_i)+d(z,y_{i+1})+d(x_i,y_{i+1})+\sum_{j\neq i} d(x_j,y_i)\nonumber\\
 & ~\geq~d(x_i,y_i)+\sum_{j\neq i} d(x_j,y_i)\label{eq:fgf}\\
& ~\geq~d_F(x_1, \ldots, x_n),\label{eq:fgf01}
\end{align}
where \eqref{eq:fgf} is obtained by a double application of the triangle inequality and \eqref{eq:fgf01} is obtained by definition of $d_F$.

It follows from \eqref{eq:fgf01} that $\sum_{i=1}^n d_F(x_1, \ldots, x_n)_i^z\geq \lfloor n/2\rfloor\, d_F(x_1, \ldots, x_n)$, which proves that $K_n^*\leq  \lfloor n/2\rfloor^{-1}$.
\end{proof}

The next proposition uses a more refined counting argument to provide an improvement of the upper bound obtained for $K_n^*$ in Proposition \ref{prop:gen-F-K}. Let us first state an immediate generalization of the hand-shaking lemma, which is folklore in graph theory.

\begin{lemma}\label{lem:graph}
Let $G=(V,E, w)$ be a weighted simple graph, where $w\colon E\to \R_+$ is the weighting function. If $f\colon V \to \R_+$ is such that $f(x)+f(y)\geq w(e)$ for every $e=\{x,y\}\in E$, then
\[
\sum_{x\in V} f(x)\, \deg_G(x)~\geq~\sum_{e\in E} w(e),
\]
where $\deg_G(x)$ is the degree of $x$ in $G$.
\end{lemma}

\begin{proposition}\label{prop:best}
For every $n\geq 2$, the best constant $K_n^*$ associated with the Fermat $n$-distance
 satisfies $K_n^*\le (4\, n-4)/(3\, n^2-4\, n).$
\end{proposition}

\begin{proof}
Let $z,x_1, \ldots, x_n, y, y_1, \ldots, y_n \in X$ be such that $y$ is a Fermat point of $\{x_1, \ldots, x_n\}$ and such that equation \eqref{eq:lmsngf0} holds for every $i\in [n]$. For any distinct $i, j\in [n]$, by the triangle inequality we have
\begin{equation}\label{eq:best01}
d(z,y_i)+d(z,y_j)+d(x_i,y_j)~\geq~d(x_i, y_i).
\end{equation}
By summing \eqref{eq:best01} over all $j\in [n]\setminus\{i\}$ we obtain
\begin{equation}\label{eq:best02}
(n-1)\, d(z,y_i)+\sum_{j\neq i} \big( d(z,y_j)+d(x_i, y_j)\big)~\geq~(n-1)\, d(x_i,y_i).
\end{equation}
By summing \eqref{eq:best02} over all $i\in [n]$ we then obtain
\begin{equation}\label{eq:best0}
2\, (n-1)\, \sum_{i=1}^n d(z,y_i)+\sum_{i=1}^n\sum_{j\neq i}d(x_j,y_i)~\geq~(n-1)\, \sum_{i=1}^n d(x_i,y_i).
\end{equation}
Let us set $S=\sum_{i=1}^n \sum_{j\neq i} d(x_j,y_i)$. We then have
\begin{align}
2\,(n-1)\, \sum_{i=1}^n d_F(x_1, \ldots, x_n)_i^z & = ~(2\, n-3)\, S+S+2\, (n-1)\sum_{i=1}^n d(z,y_i)\label{eq:best03}\\
& \geq~(2\, n-3)\, S+(n-1)\, \sum_{i=1}^n d(x_i,y_i)\label{eq:best04}\\
& =~(n-2)\, S+(n-1)\, \sum_{i=1}^n \sum_{j=1}^n d(x_j,y_i),\label{eq:best05}
\end{align}
where \eqref{eq:best03} follows by the definitions of $S$ and $d_F$, \eqref{eq:best04} follows by \eqref{eq:best0}, and \eqref{eq:best05} by the definition of $S$.

Now, on the one hand, by the definition of $d_F$ we have
\begin{equation}\label{eq:best06}
(n-1)\, \sum_{i=1}^n \sum_{j=1}^n d(x_j,y_i)~\geq~n\,(n-1)\,d_F(x_1, \ldots, x_n).
\end{equation}
On the other hand, let us fix $i\in [n]$ and set $V=\{x_1, \ldots, x_n\}\setminus\{x_i\}$. Define the function $f\colon V \to \R_+$ by $f(x_j)=d(x_j,y_i)$ for any $j\neq i$, and consider the complete weighted graph $G=(V, \binom{V}{2}, w)$ defined by $w(\{x_\ell, x_j\})=d(x_\ell, x_j)$ for any distinct $x_\ell, x_k \in V$. It follows from Lemma \ref{lem:graph} that
\begin{equation}\label{eq:best07}
(n-2)\, \sum_{j\neq i} d(x_j, y_i)~\geq~\sum_{\{x_k,x_\ell\}\in \binom{V}{2}} d(x_k,x_\ell).
\end{equation}
By summing \eqref{eq:best07} over all $i\in [n]$, we get
\begin{align}
(n-2)\, S & ~\geq~(n-2)\ \sum_{\{k,\ell\}\in \binom{n}{2}} d(x_k,x_\ell)~=~\frac{n-2}{2}\, \sum_{k=1}^n\sum_{\ell=1}^n d(x_k, x_\ell)\nonumber\\
& ~\geq ~n\, \frac{n-2}{2}\, d_F(x_1, \ldots, x_n),\label{eq:best08}
\end{align}
where \eqref{eq:best08} is obtained by definition of $d_F$. By substituting \eqref{eq:best06} and  \eqref{eq:best08} into  \eqref{eq:best05}, we finally obtain
\[
\sum_{i=1}^n d_F(x_1, \ldots, x_n)_i^z~\geq~\frac{ n\, (3\, n-4)}{4\, (n-1)}\, d_F(x_1, \ldots, x_n),
\]
which proves that $K_n^*\leq  (4\, n-4)/(3\, n^2-4\, n)$.
\end{proof}

We observe that Proposition~\ref{prop:best} provides a better upper bound than Proposition~\ref{prop:gen-F-K} for every $n\geq 2$, but the difference between these bounds converges to zero as $n$ tends to infinity. The high number of inequalities involved in the proof of Proposition \ref{prop:best} suggests that it is in general very difficult to obtain the exact value of $K_n^*$ (we have to find $x_1, \ldots, x_n, z\in X$ that turn these inequalities into equalities). However, we will now show that we can determine the value of $K_n^*$ when $d_F$ is the Fermat $n$-distance associated with the distance function in median graphs.

Recall that a \emph{median graph} is a connected undirected simple graph in which, for any triplet of vertices $u,v,w$, there is one and only one vertex $\mathbf{m}(u,v,w)$ that is at the intersection of shortest paths between any two elements among $u,v,w$. Cubes and trees are instances of median graphs. In a median graph $G=(V,E)$, the Fermat 3-distance is the function $d_{\mathbf{m}}\colon V^3\to\R_+$ defined by
\begin{equation}\label{eq:medG}
d_{\mathbf{m}}(u,v,w) ~=~ \min_{y\in V} ~ \big(d(u,y)+d(v,y)+d(w,y)\big),
\end{equation}
where $d$ denotes the usual distance function between vertices in a connected graph.

\begin{proposition}
If $G=(V,E)$ is a median graph, then the best constant $K^*$ associated with its Fermat 3-distance $d_\me$ is equal to $\frac{1}{2}$.
Moreover, the only Fermat point of $\{u,v,w\}$ is $\me(u,v,w)$.
\end{proposition}

\begin{proof}
The minimum in \eqref{eq:medG} is realized by any $y_0\in V$ that realizes the minimum of the values
\begin{equation}\label{eq:dbmed}
\big(d(u,y)+d(v,y)\big)+\big(d(w,y)+ d(u,y)\big)+\big(d(v,y)+d(w,y)\big)
\end{equation}
for $y\in V$. By definition, the vertex $y_0=\me(u,v,w)$ is on shortest paths  between any two elements among $u,v, w$, which shows that it realizes the minimum of each of the three terms in \eqref{eq:dbmed}, and hence the minimum in \eqref{eq:medG}.

It follows that
\begin{align*}
 d_\me(u,v,z) &= d(u,y_0)+d(v,y_0)+d(z,y_0)\\
& ~=~ \frac{1}{2}(d(u,y_0)+d(v,y_0)+d(z,y_0)+d(u,y_0)+d(v,y_0)+d(z,z_0))\\
& ~=~  \frac{1}{2}(d(u,v)+d(u,z)+d(v,z)),
\end{align*}
which shows that $\min_{z\in V} d_\me(u,v,z)$ is equal to $d(u,v)$, and is realized by any element $z_0$ on a shortest path between $u$ and $v$. We conclude that the minimum of
\[
d_\me(z,v,w)+d_\me(u,z,w)+d_\me(u,v,z)
\]
for $z \in V$ is realized by $z_0=\me(u,v,w)$, and is equal to $d(v,w)+d(u,w)+d(u,v)=2\, d_\me(u,v,w)$. We have proved that the best constant $K^*$ associated with $d_\me$ is $\frac{1}{2}$.
\end{proof}

\section{Examples of $n$-distances based on geometric constructions}\label{sec:geo}

In this section we introduce $n$-distances defined from certain geometric constructions and investigate their corresponding best constants. In what follows, we denote by $d$ the Euclidean distance on $\R^k$ for some integer $k\geq 2$.

The first $n$-distances we investigate are based on the following construction.

\begin{definition}\label{de:sphere}
For any $n\geq 2$ and any $x_1, \ldots, x_n \in \R^k$, we denote by $S(x_1,\ldots, x_n)$ the smallest  $(k-1)$-dimensional sphere enclosing $\{x_1, \ldots, x_n\}$. For any $i\in [n]$ and any $z\in\R^k$, we denote by $S(x_1,\ldots,x_n)_i^z$ the smallest $(k-1)$-dimensional sphere enclosing $\{x_1,\ldots,x_{i-1},z,x_{i+1},\ldots,x_n\}$.
\end{definition}

The sphere introduced in Definition~\ref{de:sphere} always exists and is unique. Moreover, it can be computed in linear time \cite{Megiddo1983,Megiddo1984} or expected linear time \cite{Welzl1991}.

When $k=2$, we have the following fact.

\begin{fact}\label{fact:trian}
Let $A, B, C$ be the vertices of a triangle in $\R^2$.
\begin{enumerate}
\item[(a)]\label{it:tr01} If $ABC$ forms an acute triangle with angles $\alpha$, $\beta$ and $\gamma$, respectively, then $S(A,B,C)$ is the circumcircle $\mathcal{C}$ of $ABC$ whose radius $R$ satisfies
\begin{equation}\label{eq:sin}
R~=~\frac{a}{2\, \sin \alpha}~=~\frac{b}{2\, \sin \beta}~=~\frac{c}{2\, \sin \gamma}{\,},
\end{equation}
where $a=d(B,C)$, $b=d(A,C)$, and $c=d(A,B)$. Let $A^*$ be one of the two points of the circle $\mathcal{C}$ that is on the bisector of $BC$. Then the perimeter of the triangle $ABC$ strictly decreases as $A$ moves along $\mathcal{C}$ from $A^*$ to $B$ (or from $A^*$ to $C$).

\item[(b)]\label{it:tr02} If $ABC$ is obtuse in $A$, then $S(A,B,C)$ contains $B$ and $C$, and its diameter is equal to $a$.
\item[(c)]\label{it:tr03} It follows from (a) and (b) that the radius  $R$ of $S(A,B,C)$ satisfies
\begin{equation}\label{eq:triv}
R~\geq~\max\big\{\frac{a}{2},{\,}\frac{b}{2}, {\,}\frac{c}{2} \big\}.
\end{equation}

\end{enumerate}
\end{fact}

\begin{proposition}[Radius of $S(x_1, \ldots, x_n)$ in $\R^2$]\label{prop:radius}
For any $n \geq 2$, the map $d_r\colon(\R^2)^n \to \R_+$ that associates with any $(x_1, \ldots, x_n) \in (\R^2)^n$ the radius of $S(x_1, \ldots, x_n)$ is an $n$-distance for which we have $K_n^*=\frac{1}{n-1}$.
\end{proposition}

\begin{proof}
Let us show that the map $d_r$ satisfies the simplex inequality for $K_n=\frac{1}{n-1}$. Since $d_r$ is a continuous function, we can assume that its arguments are pairwise distinct.

Consider first the case where $n=2$. For any distinct $A, B\in \R^2$, we have $d_r(A,B)=\frac{1}{2}\, d(A,B)$, which proves that the simplex inequality holds for $n=2$.

Suppose now that $n=3$ and let us show that, for any $A,B,C,Z\in\R^2$, with $A,B,C$ pairwise distinct, we have
\begin{equation}\label{eq:juu}
2{\,}d_r(A,B,C)~\leq~ d_r(Z,B,C)+d_r(A,Z,C)+d_r(A,B,Z).
\end{equation}
Set $a=d(B,C)$, $b=d(A,C)$, and $c=d(A,B)$. By \eqref{eq:triv} we have
\begin{equation}\label{eq:trivbis}
d_r(Z,B,C)~\geq~ \frac{a}{2}~, \quad d_r(A,Z,C)~\geq~ \frac{b}{2}~, \quad d_r(A,B,Z)~\geq~ \frac{c}{2}~,
\end{equation}
and hence
\begin{equation}\label{eq:dcsdpgf}
d_r(Z,B,C)+d_r(A,Z,C)+d_r(A,B,Z)~\geq~\frac{a+b+c}{2} ~\geq ~\max\{a,b,c\}.
\end{equation}
Suppose first that $ABC$ is not acute, assuming for instance that $\beta\geq \frac{\pi}{2}$. Then $2{\,}d_r(A,B,C)=b$, and then \eqref{eq:juu} immediately follows from \eqref{eq:dcsdpgf}. Suppose now that $ABC$ is acute, with circumcircle $\mathcal{C}$, and consider the triangle $A'BC$, with sides $a$, $b'$, $c'$, such that $A'\in\mathcal{C}$ and ${\sphericalangle{}A'BC}=\frac{\pi}{2}$. By Fact \ref{fact:trian} (a) we have
$$
\frac{a+b+c}{2} ~\geq ~ \frac{a+b'+c'}{2} ~\geq ~ b' ~=~ 2{\,}d_r(A',B,C) ~=~ 2{\,}d_r(A,B,C),
$$
and then again \eqref{eq:juu} follows from \eqref{eq:dcsdpgf}. Finally, the equality is obtained in \eqref{eq:juu} by taking $A\neq B=C=Z$.

We now prove the general case where $n\geq 3$.  Let $A_1, \ldots, A_n, Z\in \R^2$, with $A_1, \ldots, A_n$ pairwise distinct. It is a known fact \cite{Chry1884} that either there are $j,k\in [n]$ such that $A_j$ and $A_k$ are distinct and
$$
S(A_1, \ldots, A_n)~=~S(A_j,A_k)
$$
or there are $j,k,\ell \in [n]$ such that $A_j$, $A_k$, and $A_\ell$ are distinct and
$$
S(A_1, \ldots, A_n)~=~S(A_j,A_k,A_\ell).
$$
Let us consider the latter case (the proof in the former case can be dealt with similarly). On the one hand, using \eqref{eq:triv} it is easy to see that
\begin{equation}\label{eqjn:predn}
d_r(A_1, \ldots, A_n)_i^Z ~\geq ~ d_r(A_1, \ldots, A_n),\qquad i\notin\{j,k,\ell\}.
\end{equation}
On the other hand, the following inequalities hold:
\begin{eqnarray*}
d_r(A_1, \ldots, A_n)_j^Z &\geq & d_r(Z,A_k,A_{\ell}),\\
d_r(A_1, \ldots, A_n)_k^Z &\geq & d_r(A_j,Z,A_{\ell}),\\
d_r(A_1, \ldots, A_n)_{\ell}^Z &\geq & d_r(A_j,A_k,Z).
\end{eqnarray*}
Indeed, $S(A_1, \ldots, A_n)_j^Z$ encloses the points $Z$, $A_k$, and $A_{\ell}$ and hence cannot have a radius strictly smaller than that of $S(Z,A_k,A_{\ell})$.

Adding up these inequalities and then using \eqref{eq:juu}, we obtain
\begin{eqnarray}
\lefteqn{d_r(A_1, \ldots, A_n)_j^Z + d_r(A_1, \ldots, A_n)_k^Z + d_r(A_1, \ldots, A_n)_{\ell}^Z}\nonumber\\
& \geq &  d_r(Z,A_k,A_{\ell}) + d_r(A_j,Z,A_{\ell}) + d_r(A_j,A_k,Z)\label{eq:3-w1}\\
& \geq & 2{\,}d_r(A_j,A_k,A_{\ell}) ~=~ 2{\,}d_r(A_1, \ldots, A_n).\nonumber
\end{eqnarray}

Combining \eqref{eqjn:predn} with \eqref{eq:3-w1}, we finally obtain
$$
\sum_{i=1}^n d_r(A_1, \ldots, A_n)_i^Z ~\geq ~ (n-1)\, d_r(A_1, \ldots, A_n),
$$
which proves that $K_n^*\leq\frac{1}{n-1}$. To prove that $K_n^*=\frac{1}{n-1}$, just consider $A_2=\cdots=A_n=Z$ and $A_1\neq A_2$.
\end{proof}

\begin{proposition}[Area bounded by $S(x_1, \ldots, x_n)$ in $\R^2$]\label{prop:area}
For any $n \geq 3$, the map $d_s\colon(\R^2)^n \to \R_+$ that associates with any $(x_1, \ldots, x_n) \in (\R^2)^n$ the surface area bounded by $S(x_1, \ldots, x_n)$ is an $n$-distance for which we have $K_n^*=(n-\frac{3}{2})^{-1}$.
\end{proposition}

\begin{proof}
Let us show that the map $d_s=\pi{\,}d_r^2$ satisfies the simplex inequality with constant $K_n=(n-\frac{3}{2})^{-1}$. Since $d_r$ is continuous, we can assume that its arguments are pairwise distinct.

Consider first the case where $n=3$ and let us show that, for any $A,B,C,Z \in \R^2$, with $A,B,C$ pairwise distinct, we have
\begin{equation}\label{eq:sqr}
d_r(A,B,C)^2 ~\leq~ \frac{2}{3}\, \big(d_r(Z,B,C)^2+d_r(A,Z,C)^2+d_r(A,B,Z)^2\big).
\end{equation}

If the triangle $ABC$ is acute, then we may assume for instance that $\frac{\pi}{3}\leq\alpha\leq\frac{\pi}{2}$, which implies $\frac{\sqrt{3}}{2}\leq\sin\alpha\leq 1$. Using \eqref{eq:sin}, we then have
\begin{equation}\label{eq:bftt}
d_r(A,B,C)^2~\leq~\frac{a^2}{3}~\leq~\frac{2}{3}\, (\frac{a^2}{4}+\frac{a^2}{4})~\leq~\frac{2}{3}\, (\frac{a^2}{4}+\frac{b^2}{4}+\frac{c^2}{4}),
\end{equation}
where the latter inequality holds by the law of cosines. We then obtain \eqref{eq:sqr} by combining \eqref{eq:trivbis} with \eqref{eq:bftt}.

If $ABC$ is obtuse in $C$, then $d_r(A,B,C)=\frac{c}{2}$. Using the triangle inequality and the square and arithmetic mean inequality, we also have
$$
\frac{a^2+b^2}{2}~\geq ~\Big(\frac{a+b}{2}\Big)^2~\geq ~\frac{c^2}{4}{\,}.
$$
Combining these observations with \eqref{eq:trivbis}, we obtain
\begin{eqnarray*}\label{eq:fer}
\lefteqn{\frac{2}{3}\,  \big(d_r(Z,B,C)^2+d_r(A,Z,C)^2+d_r(A,B,Z)^2\big)}\\
& \geq & \frac{2}{3}\, \Big(\frac{a^2}{4}+\frac{b^2}{4}+\frac{c^2}{4}\Big) ~\geq ~ \frac{2}{3}~\frac{3}{8}\, c^2 ~=~ \Big(\frac{c}{2}\Big)^2~=~\,d_r(A,B,C)^2.
\end{eqnarray*}

To see that the general case where $n\geq 3$ also holds, it suffices to proceed as in the proof of Proposition~\ref{prop:radius}. This shows that $K_n^*\leq (n-\frac{3}{2})^{-1}$. To prove that $K_n^*=(n-\frac{3}{2})^{-1}$, just consider $A_1\neq A_2$ and $A_3= \cdots = A_n=Z=({A_1+A_2})/{2}$, where $({A_1+A_2})/{2}$ is the midpoint of $A_1$ and $A_2$.
\end{proof}

\begin{remark}
The map $d_s$ defined in Proposition~\ref{prop:area} can be naturally extended to the case where $n=2$. However, in this case $d_s$ no longer satisfies condition (i) and hence is not a $2$-distance. Indeed, for any $A,B,Z \in \R^2$, with $A,B$ distinct, we have
$$
d_s(A,B) ~\leq ~ 2{\,}\big(d_s(A,Z)+d_s(Z,B)\big),
$$
or equivalently,
$$
d(A,B)^2 ~\leq ~ 2{\,}d(A,Z)^2+2{\,}d(Z,B)^2,
$$
where the constant $2$ is optimal (take $A$ and $B$ distinct and $Z=(A+B)/2$). To see that this inequality holds, set $A=(0,0)$, $B=(b,0)$, and $Z=(x,y)$. Then, the inequality becomes
$$
b^2 ~\leq ~ 2{\,}(x^2+y^2)+2{\,}(x-b)^2+2{\,}y^2,
$$
which always holds because it is algebraically equivalent to
$$
(2x-b)^2+4y^2 ~\geq ~0.
$$
\end{remark}

\begin{remark}\label{rem:conj}
In an attempt to generalize the previous two propositions to $\R^k$ ($k\geq 2$), we may consider the following open questions:
\begin{enumerate}
  \item[(a)] Prove (or disprove) that Proposition~\ref{prop:radius} still holds in $\R^k$.
  \item[(b)] Prove (or disprove) that, for any $n \geq 3$, the map $d_v\colon(\R^k)^n \to \R_+$ that associates with any $(x_1, \ldots, x_n) \in (\R^k)^n$ the $k$-dimensional volume bounded by $S(x_1, \ldots, x_n)$ is an $n$-distance for which we have $K_n^*=(n-2+2^{1-k})^{-1}$.
\end{enumerate}
Note that the problem in (b) above is motivated by the fact that the corresponding simplex inequality with $K_n=(n-2+2^{1-k})^{-1}$ holds when $x_1$ and $x_2$ are distinct and $x_3=\cdots=x_n=z$ is the midpoint of $x_1$ and $x_2$.
\end{remark}

We now show that counting the number of different directions defined by pairs of distinct elements among $n$ points in the plane defines an $n$-distance.

For any distinct $x,y\in \R^2$, we denote by $\overline{xy}$ the direction $\pm (x-y)/||x-y||$. Here we assume that $\overline{xy}$ and $\overline{yx}$ represent the same direction.

\begin{proposition}[Number of directions in $\R^2$]\label{prop:dir}
For any $n\geq 3$, the map $d_n\colon (\R^2)^n \to \R_+$ that associates with any $(x_1, \ldots, x_n)\in (\R^2)^n$ the cardinality $|\Delta|$ of the set
$$
\Delta ~=~ \big\{\overline{x_ix_j} \mid i,j\in [n]~\text{and} ~ x_i\neq x_j\big\}
$$
is an $n$-distance for which we have
$
\frac{1}{n-2+\frac{2}{n}} \leq K_n^* < \frac{1}{n-2}{\,}.
$
\end{proposition}

\begin{proof}
Let $x_1, \ldots, x_n, z \in \R^2$. For any $i\in [n]$, let
$$
\Delta_i ~=~ \{\overline{x_jx_k} \mid j,k\in [n]\setminus\{i\}~\text{and} ~ x_j\neq x_k\}.
$$
On the one hand, we clearly have $|\Delta_i|\leq d_n(x_1, \ldots, x_n)_i^z$ for every $i\in [n]$. On the other hand, it is easy to see that each direction in $\Delta$ is counted at least $(n-2)$ times in the sum $\sum_{i=1}^n|\Delta_i|$. From these observations it follows that
\begin{equation}\label{eqn:gdt}
(n-2){\,}d_n(x_1, \ldots, x_n) ~=~ (n-2){\,}|\Delta| ~\leq ~ \sum_{i=1}^n|\Delta_i|~\leq ~ \sum_{i=1}^n d_n(x_1, \ldots, x_n)_i^z,
\end{equation}
which proves that $K_n^*\leq \frac{1}{n-2}$.

We now show by contradiction that the latter inequality is strict. Assume that there exist $x_1, \ldots, x_n, z\in \R^2$ such that
$$
(n-2) \, d_n(x_1, \ldots, x_n) ~=~\sum_{i=1}^n d_n(x_1, \ldots, x_n)_i^z.
$$
It follows that for these points we can replace both inequalities in \eqref{eqn:gdt} with equalities. The first equality then means that each direction in $\Delta$ is counted exactly $(n-2)$ times in the sum $\sum_{i=1}^n|\Delta_i|$. It is easy to see that this condition also means that no three of the points $x_1,\ldots,x_n$ are collinear. 
Let us now consider the second inequality. Since $|\Delta_i|\leq d_n(x_1, \ldots, x_n)_i^z$ for every $i\in [n]$, we must have $|\Delta_i|= d_n(x_1, \ldots, x_n)_i^z$ for every $i\in [n]$.
Suppose first that $n\geq 4$. It follows from the latter condition that both sets $\{x_2,\ldots,x_n\}$ and $\{z,x_2,\ldots,x_n\}$ generate the same number of directions. Since no three of the points $x_2,\ldots,x_n$ are collinear, we should have $z=x_{\ell}$ for some $\ell\in\{2,\ldots,n\}$. But then we have $|\Delta_{\ell}|< d_n(x_1, \ldots, x_n)_{\ell}^z$, a contradiction. A similar contradiction can be easily reached when $n=3$.

Let us now establish the lower bound for $K_n^*$. Let $x_1, \ldots, x_n$ be pairwise distinct and placed clockwise on the unit circle. Let also $z=x_1$. Then we have
\[
d_n(x_1, \ldots, x_n)~=~{n\choose 2} \quad\text{and} \quad d_n(x_1, \ldots, x_n)_{i}^z~=~
\begin{cases}
{n\choose 2} & \text{if $i=1$},\\
{n-1\choose 2} & \text{if $i\neq 1$},
\end{cases}
\]
and hence
\begin{eqnarray*}
\sum_{i=1}^n d_n(x_1, \ldots, x_n)_{i}^z &=& {n\choose 2}+(n-1){n-1\choose 2} ~=~ \big(n-2+\frac{2}{n}\big){\,}{n\choose 2}\\
&=& \big(n-2+\frac{2}{n}\big){\,}d_n(x_1, \ldots, x_n),
\end{eqnarray*}
which completes the proof.
\end{proof}

\begin{remark}
An $n$-distance $d\colon (\R^k)^n\to\R_+$ is said to be \emph{homogeneous of degree $q\geq 0$} if, for any $t>0$, we have
$$
d(tx_1,\ldots,tx_n) ~=~ t^q{\,} d(x_1,\ldots,x_n){\,},\qquad x_1,\ldots,x_n\in \R^k.
$$
This means that under any dilation $x\mapsto tx$, the $n$-distance $d$ is magnified by the factor $t^q$. Since a distance on $\R^k$ usually represents a linear dimension, we could expect any $n$-distance on $\R^k$ to be homogeneous of degree 1. This is for instance the case for the $n$-distance defined in Proposition~\ref{prop:radius}. Surprisingly enough, the $n$-distances defined in Examples~\ref{ex:dra}, \ref{ex:car}, and Proposition~\ref{prop:dir} are homogeneous of degree $0$, that is, invariant under any dilation. Also, the $n$-distance defined in Proposition~\ref{prop:area} is homogeneous of degree $2$.
\end{remark}

\section{A generalization of the concept of $n$-distance}\label{sec:g}

The concept of $n$-distance as defined in Definition~\ref{de:n-dist} can naturally be generalized by relaxing condition (i) as follows.

\begin{definition}
Let $g\colon\R^n_+\to\R_+$ be a symmetric function, i.e., invariant under any permutation of its arguments. We say that a function $d\colon X^n\to\R_+$ is a \emph{$g$-distance} if it
satisfies conditions (ii), (iii), and
$$
d(x_1,\ldots,x_n) ~\leqslant ~ g\big(d(x_1,\ldots,x_n)_{1}^z{\,},\ldots,{\,}d(x_1,\ldots,x_n)_{n}^z\big)
$$
for all $x_1,\ldots,x_n,z\in X$.
\end{definition}

In view of Proposition~\ref{prop:buit-d}, it is natural to require $d+d'$, $\lambda{\,}d$, and $\frac{d}{1+d}$ to be $g$-distances whenever so are $d$ and $d'$. The following proposition provides sufficient conditions on $g$ for these properties to hold. Recall that a function $g\colon \R^n_+ \to \R$ is \emph{positively homogeneous} if $g(\lambda\, \mathbf{r})=\lambda\, g(\mathbf{r})$ for all $\mathbf{r} \in \R^n_+$ and all $\lambda >0$. It is said to be \emph{superadditive} if $g(\mathbf{r}+\mathbf{s})\geq g(\mathbf{r})+g(\mathbf{s})$ for every $\mathbf{r}, \mathbf{s}\in \R^n_+$. Also, it is \emph{additive} if $g(\mathbf{r}+\mathbf{s})=g(\mathbf{r})+g(\mathbf{s})$ for every $\mathbf{r}, \mathbf{s}\in \R^n_+$.

\begin{proposition}\label{prop:gen}
Let $g\colon\R^n_+\to\R_+$ be a symmetric function, and let $d,d'\colon X^n \to \R_+$ be $g$-distances. The following
assertions hold.
\begin{enumerate}
\item[(a)] If $g$ is positively homogeneous, then $\lambda\, d$ is a $g$-distance for every $\lambda >0$.
\item[(b)] If $g$ is superadditive, then $d+d'$ is a $g$-distance.
\item[(c)] If $g$ is both positively homogeneous and superadditive, then it is concave.
\item[(d)] The function $g$ is additive if and only if there exists $\lambda\geq 0$ such that
    \begin{equation}\label{eq:add}
    g(\mathbf{r}) ~=~ \lambda{\,}\sum_{i=1}^n r_i{\,}, \qquad \mathbf{r}=(r_1,\ldots,r_n)\in \R^n_+.
    \end{equation}
\item[(e)] If $g$ satisifies \eqref{eq:add} for some $\lambda\geq 1$, then $\frac{d}{1+d}$ is a $g$-distance.
\end{enumerate}
\end{proposition}

\begin{proof}
(a) and (b) follow from the definitions.

(c) For any $\lambda\in [0,1]$, we have \[g(\lambda \mathbf{r}+(1-\lambda)\mathbf{s})~\leq~g(\lambda \mathbf{r})+g((1-\lambda)\mathbf{s})~=~\lambda
g(\mathbf{r})+(1-\lambda)g(\mathbf{s}),\]
where the inequality follows from superadditivity and the equality from positive homogeneity.

(d) The sufficiency is trivial. To see that the necessity holds, note that $g$ is additive and bounded from below (since it ranges in $\R_+$) and hence it is continuous and there exist $\lambda_1,\ldots,\lambda_n\in\R$ such that $g(\mathbf{r}) = \sum_{i=1}^n \lambda_ir_i$ ; see \cite[Cor.~2, p.~35]{Aczel2008}. The result then follows from the symmetry of $g$.

(e) Let $x_1, \ldots, x_n, z\in X$ and set $d=d(x_1, \ldots, x_n)$ and $d_i=d(x_1, \ldots, x_n)_i^z$ for every $i\in [n]$. Since $\lambda \geq 1$, we have $\lambda\, r/(1+\lambda\, r) \leq \lambda\, r/(1+ r)$ for every $r\geq 0$. It then follows that
\[
\frac{1}{1+d}~\leq~\sum_{i=1}^n \frac{\lambda\, d_i}{1+\lambda\, d_i}~\leq~\sum_{i=1}^n \frac{\lambda\, d_i}{1+d_i}{\,},
\]
where the first inequality follows from Lemma \ref{lem:sum} and the fact that $d$ is a $g$-distance.
\end{proof}

\section{Conclusion and further research}

In this paper we have introduced and discussed the concept of $n$-distance as a natural generalization of the concept of distance to functions of $n\geq 2$ variables. There are two key features in this generalization: one is an $n$-ary version of the identity of indiscernibles, and the other is the simplex inequality, which is a natural generalization of the triangle inequality. We have observed that any $n$-distance $d$ has an associated best constant $K_n^*\in\left]0,1\right]$ satisfying inequality \eqref{eq:ntr}. Also, we have provided many natural examples of $n$-distances, and have shown that searching for their associated best constant may be mathematically challenging and may sometimes require subtle arguments. The examples we have discussed might suggest that we have $K_n^*<1$ for any $n$-distance. The following example, which was communicated to us by Roberto Ghiselli Ricci \cite{Ricci2016}, shows that this is not the case.

\begin{example}
Let $n\geq 3$ and $a\in\R$. Let also $\mathcal{A}(a,n)$ be the set of $n$-tuples whose components are consecutive elements of arithmetic progressions with common difference $a$. Consider the map $d_n\colon \R^n \to \R_+$ defined as
\[
d_n(x_1, \ldots, x_n)=\left\{\begin{array}{ll}
0 & \text{ if } x_1=\cdots=x_n,\\
1 & \text { if } (x_1, \ldots, x_n)\in \mathcal{A}(a,n) \text{ for some } a\neq 0,\\
\frac{1}{n} & \text{ otherwise}.
\end{array}\right.
\]
We prove that $d_n$ is an $n$-distance for which we have $K_n^*=1$. Conditions (ii) and (iii) are easily verified. To see that condition (i) holds, consider $x_1, \ldots, x_n, z\in \R$. First assume that $d_n(x_1, \ldots, x_n)=\frac{1}{n}$. There is at most one $i\in [n]$ such that $d_n(x_1, \ldots, x_n)_i^z=0$. Thus, we obtain
\[
\sum_{i=1}^n d_n(x_1, \ldots, x_n)_{i}^z~\geq~\frac{n-1}{n}~\geq~d_n(x_1, \ldots, x_n).
\]
Assume now that $d(x_1, \ldots, x_n)=1$. It follows that  $d_n(x_1, \ldots, x_n)_i^z\geq\frac{1}{n}$ for all $i\in [n]$, which shows that the simplex inequality holds in that case as well. To prove that $K_n^*=1$, just consider $x_1=1, x_2=2, \ldots, x_n=n$, and $z=-1$.\qed
\end{example}

We also observe that certain $n$-distances cannot be constructed from the concept of multidistance as defined by Mart\'{\i}n and Mayor~\cite{MarMay11} (see Remark~\ref{rem:multdist}). Instances of such $n$-distances are given, e.g., in Propositions~\ref{prop:area} and \ref{prop:dir}.

We conclude this paper by proposing a few topics for further research.
\begin{enumerate}
\item[(a)] Improve the bounds for the best constant associated with the Fermat $n$-distance (at least in some given proper metric spaces).

\item[(b)] Consider and solve the problems stated in Remark~\ref{rem:conj}.

\item[(c)] Investigate properties of topological spaces based on $n$-metric spaces. On this issue we observe that in \cite{Mustafa2006} the authors introduced a stronger version of $3$-metric space called \emph{$G$-metric space} (see also \cite{Khamsi2015}). It is shown that there is a natural metric space associated with any $G$-metric space. Finding an appropriate generalization of the notion of $G$-metric space as a stronger version of $n$-metric space and investigating its topological properties seems to be an interesting topic of research.
\end{enumerate}

\section*{Acknowledgments}

This research is supported by the internal research project R-AGR-0500 of the University of Luxembourg.

\end{document}